\documentclass{amsart}
\usepackage{amsmath}%
\usepackage{amsfonts}%
\usepackage{amssymb}%
\usepackage{graphicx}
\newtheorem{theorem}{Theorem}[section]

\newtheorem{corollary}[theorem]{Corollary}

\newtheorem{definition}[theorem]{Definition}

\newtheorem{lemma}[theorem]{Lemma}
\newtheorem{sublemma}[theorem]{Sublemma}

\newtheorem{proposition}[theorem]{Proposition}
\newtheorem{remark}[theorem]{Remark}

\numberwithin{equation}{section}
\def\Xint#1{\mathchoice
{\XXint\displaystyle\textstyle{#1}}%
{\XXint\textstyle\scriptstyle{#1}}%
{\XXint\scriptstyle\scriptscriptstyle{#1}}%
{\XXint\scriptscriptstyle\scriptscriptstyle{#1}}%
\!\int}
\def\XXint#1#2#3{{\setbox0=\hbox{$#1{#2#3}{\int}$}
\vcenter{\hbox{$#2#3$}}\kern-.5\wd0}}

\def\dashint{\Xint-}

\newcommand{\R}{\mathbb{R}}

\newcommand{\N}{\mathbb{N}}
\newcommand{\Z}{\mathbb{Z}}

\newcommand{\M}{\mathcal{M}}
\newcommand{\F}{\mathcal{F}}
\newcommand{\Ha}{\mathcal{H}}

\newcommand{\W}[2]{{W^{1,#1}(#2)}}

\newcommand{\Ne}[2]{N^{1,#1}(#2)}

\newcommand{\Nem}[3]{N^{1,#1}(#2;#3)}

\newcommand{\Wem}[3]{W^{1,#1}(#2;#3)}

\newcommand{\spt}{\operatorname{spt}}
\newcommand{\Mod}{\operatorname{Mod}}
\newcommand{\dist}{\operatorname{dist}}

\newcommand{\Lip}{\operatorname{Lip}}
\newcommand{\LIP}{\operatorname{LIP}}
\newcommand{\ud}{\mathrm {d}}

\newcommand{\Cp}{\operatorname{Cap}}

\title{Path and quasihomotopy for Sobolev maps between manifolds}
\thanks{This research was conducted in the University of Jyv\"askyl\"a and at IMPAN, Warsaw.}
\author{Elefterios Soultanis}
\date{\today}

\address{ul. \'Sniadeckich 8,\newline
	\indent 00-656 Warszawa}
\email{elefterios.soultanis@gmail.com}%

\begin{document}

\begin{abstract}
	We study the relationship between quasihomotopy and path homotopy for Sobolev maps between manifolds. We employ singular integrals on manifolds to show that, in the critical exponent case, path homotopy implies quasihomotopy -- and observe the rather surprising fact that $n$-quasihomotopic maps need not be path homotopic. We also study the case where the target is an aspherical manifold, e.g. a manifold with nonpositive sectional curvature, and the contrasting case of the target being a sphere.
\end{abstract}

\maketitle

\section{Introduction}

Let $M$ and $N$ be compact Riemannian manifolds with $n=\dim M\ge 2$. The study of harmonic and $p$-harmonic maps between $M$ and $N$ naturally leads to questions about homotopies between finite energy Sobolev maps \cite{eel64, eel78, eel88, ver12, pig09}.

However classical homotopy is incompatible with Sobolev maps: on one hand Sobolev maps need not be continuous, and on the other classical homototopy classes are not stable under convergence in the Sobolev norm. Indeed, an easy example by B. White \cite{whi86} showed that the identity map $S^3\to S^3$ is homotopic to maps of arbitrarily small energy, whilst not being homotopic to a constant map.

F. Burstall, in \cite{bur84}, studied energy minimization within classes of maps with prescribed 1-homotopy class, and White \cite{whi86} introduced the notion of \emph{$d$-homotopy} for an integer $d\le n=\dim M$.

\medskip\noindent\emph{Two maps $u,v\in \Wem pMN$ are $d$-homotopic, $d<p$, if the restrictions of $u$ and $v$ to a $d$-skeleton of a generic triangulation of $M$ (which are continuous by the Sobolev embedding theorem) are classically homotopic.}

\medskip\noindent White proved \cite{whi86, whi88} that Sobolev maps $u\in \Wem pMN$ ($p\le n$) have a well defined $(\lfloor p\rfloor-1)$-homotopy type (i.e. the homotopy class of the restriction of $u$ does not depend on the generic $(\lfloor p\rfloor-1)$-dimensional skeleton) that is stable under weak convergence in $\Wem pMN$, and therefore well suited for variational minimization problems.

Connections of of $d$-homotopy with the topology of the Sobolev space $\Wem pMN$ are already visible in \cite{whi86}. The notion of \emph{path homotopy}, introduced by H. Brezis and Y. Li in \cite{bre01} utilizes this idea.

\medskip\noindent\emph{Two maps $u,v\in \Wem pMN$ ($1<p<\infty$) are path homotopic if there exists a continuous path $h\in C([0,1]; \Wem pMN)$ joining $u$ and $v$.} 

\medskip\noindent They proved \cite[Theorem 0.2]{bre01} that $\Wem pMN$ is always path connected when $1<p<2$, while a deep result of Hang and Lin \cite{han03} states that, for $1<p<n$, \emph{two maps $u,v\in \Wem pMN$ are path homotopic if and only if they are $(\lfloor p\rfloor-1)$-homotopic.}

When $p=n$ this equivalence does not remain valid. Instead, Sobolev maps $u\in \Wem nMN$ have well-defined homotopy classes (due to the density of Sobolev maps \cite{uhl83} and a result of White \cite{whi86}, see Theorem \ref{white}.) When $p>n$ the Sobolev embedding implies that Sobolev maps are continuous and indeed by results in Appendix A in \cite{bre01} path homotopy is equivalent to classical homotopy.

\bigskip\noindent With the emergence of analysis on metric spaces (see \cite{haj96, hei98, hei01, sha00} and the monographs \cite{bjo11,HKST07}) the study of energy minimization problems between more general spaces has become viable. The first steps in this direction were taken by N. Korevaar and R. Schoen \cite{kor93} -- who studied the existence of minimizers of 2-energy in homotopy classes of maps from a manifold to a nonpositively curved metric space (see \cite{bri99}) -- and J. Jost \cite{jost94,jost95,jost96,jost97} who studied the related problem of minimizing 2-energy in equivariance classes of maps from $(1,2)$-Poincar\'e space spaces to nonpositively curved metric spaces.

In the more general setting both $d$-homotopy and path homotopy become problematic. The lack of triangulations in metric spaces on the one hand, and the fact that the topology of Newton Sobolev spaces $\Nem pXY$  depends on the embedding of $Y$ into a Banach space (see \cite{haj07}) on the other, make both notions of homotopy difficult to work with. 

In \cite{teri1}, for the purpose of studying minimizers of $p$-energy in homotopy classes of maps from a $(1,p)$-Poincar\'e space to a nonpositively curved metric space a third notion, called \emph{$p$-quasihomotopy}, was introduced. Here we state the definition for manifolds. It is based on the known fact that Sobolev maps $u\in \Wem pMN$ have $p$-quasicontinuous representatives, i.e. \emph{for every $\varepsilon>0$ there is an open set $E\subset M$ with $\Cp_p(E)<\varepsilon$ so that $u|_{M\setminus E}$ is continuous.} Quasicontinuity may be seen as a refinement of the almost continuity of measurable maps.

\medskip\noindent\emph{Two quasicontinuous representatives $u,v\in \Wem pMN$ ($1<p<\infty$) are $p$-quasihomotopic if there is a map $H:M\times [0,1]\to N$ with the following property: for any $\varepsilon>0$ there is an open set $E\subset M$ with $\Cp_p(E)<\varepsilon$ so that $H|_{M\setminus E\times [0,1]}$ is a (continuous) homotopy between $u|_{M\setminus E}$ and $v|_{M\setminus E}$.}

\medskip\noindent Capacity is a much finer measure of smallness than the Lebesgue measure; a set $E\subset M$ of zero $p$-capacity has Hausdorff dimension at most $n-p$, and sets of small $p$-capacity have small Hausdorff content, \[ \Cp_p(E)\le c(n,p,q)\Ha^{n-q}_\infty(E)\textrm{ for any } 1<q<p   \] (Theorem 5.3 in \cite{mal03}.) Thus, while quasihomotopy allows for discontinuities, it does so in a sense a minimal amount, preserving \emph{some} amount of topology. For example, a set of zero $p$-capacity, $p>1$, does not separate a space, whereas a set of measure zero may. There is also a $p$-quasicontinuous counterpart to the fact that if the preimage of a point of a continuous function (from a connected space) is nonempty and open, then the function must be constant (see Lemma 5.3 in \cite{teri1}).


As such, $p$-quasihomotopy is a natural relaxation of classical homotopy to encompass Sobolev maps. Indeed, under the additional assumption that the target space has hyperbolic universal cover there always exists minimizers of $p$-energy in quasihomotopy classes in the metric setting, see Theorem 1.1. in \cite{teri2}.

\medskip\noindent When $p>n$ the fact any nonempty set has $p$-capacity $\ge \varepsilon_0$ for some small number $\varepsilon_0$ implies that $p$-quasihomotopy coincides with classical homotopy, and thus with path homotopy.

However when $1<p<n$ the notion of $p$-quasihomotopy turns out to differ from the other two. Theorem 1.4 in \cite{teri1} states that when $1<p<n$, if $u,v\in \Wem pMN$ are $p$-quasihomotopic then they are path homotopic. The proof in fact yields more: \emph{if $1<p\le n$ and $u,v\in \Wem pMN$ are $p$-quasihomotopic then $u$ and $v$ are $d$-homotopic, where $d=\lceil p\rceil-1$ is the largest integer $<p$.} Since $\lfloor p\rfloor-1<\lceil p\rceil-1$ unless $p$ is an integer it is expected that path homotopic maps need not be quasihomotopic. Indeed the constant map and \[ x\mapsto \frac{x}{|x|}\in \Wem p{B^2}{S^1},\ 1<p<2 \] are path homotopic but not $p$-quasihomotopic (see Section 4.2 in \cite{teri1}).

The first main theorem in this paper considers the remaining case $p=n$.
\begin{theorem}\label{main}
	Let $M$ and $N$ be smooth compact Riemannian manifolds, with $n=\dim M$. If two maps $f,g\in W^{1,n}(M;N)$ are path homotopic then they are $n$-quasihomotopic.
\end{theorem}

The relationships between path-, quasi-, and $d$-homotopy are summarized in the table below.

\begin{center}
	\begin{tabular}{|c | c l c |} 
		\hline
		$\Wem pMN$ & &  & \\ [1ex] 
		\hline 
		$1<p<n$	& $p$-quasihomotopy  $\Rightarrow$& $([p]-1)$-homotopy &$\Leftrightarrow$ path homotopy   \\ [1ex]
		\hline
		$p=n$ & path homotopy $\Rightarrow$&  $p$-quasihomotopy &$\Rightarrow$ $(n-1)$-homotopy \\ [1ex]
		\hline
		$p>n$ & $p$-quasihomotopy $\Leftrightarrow$& homotopy &$\Leftrightarrow$ path homotopy \\ [1ex]
		\hline
	\end{tabular}
\end{center}

Surprisingly, the converse of Theorem \ref{main} fails. Namely it can happen that two maps $f,g\in \Wem nMN$ are $n$-homotopic but not path homotopic. An example to this effect is given in Corollary \ref{sn}. It is noteworthy that in the example the target has the rational homology type of a sphere (in this case it is in fact a sphere) in light of the discussion in \cite{gol12} (see in particular Theorems 1.4 and 1.5 there). An $n$-manifold $M$ is a rational homology sphere if
\begin{align*}
H^k_{dR}(M)=\left\{
\begin{array}{ll}
0 &, k\ne 0,n \\
\Z &, k=0,\ k=n,
\end{array}
\right.
\end{align*}
where $H^k_{dR}(M)$ denotes the \emph{de Rham cohomology } of $M$.

\bigskip\noindent For generic manifolds $M,N$, particularly rational homology sphere targets, the implications between path- and quasihomotopy depend on $p$.

In contrast, for \emph{aspherical} target manifolds the situation is simpler. An $m$-manifold $N$ is apsherical if the homotopy groups $\pi_k(N)$ vanish for all $k\ge 2$. Using Whiteheads theorem (Theorem 4.5 in \cite{hat02}) aspherical manifolds may be characterized as those with contractible universal cover. Aspherical manifolds include, as an important subclass, manifolds of nonpositive sectional curvature.
	
For general $p\in (1,\infty)$ we have the following theorem.
\begin{theorem}\label{pqpath}
	Suppose $M$ and $N$ are compact smooth Riemannian manifolds, $N$ aspherical and $1<p<\infty$. If two maps $u,v\in \Wem pMN$ are $p$-quasihomotopic then they are path homotopic.
\end{theorem}
	
When $p\ge 2$ we can say more.
\begin{theorem}\label{main2}
	Let $2\le p <\infty$, $M,N$ be smooth compact Riemannian manifolds, $N$ being aspherical. Then two maps $f,g\in \Wem pMN$ are path homotopic if and only they are $p$-quasihomotopic.
\end{theorem}
The restriction $p\ge 2$ is essential. Indeed by Theorem 0.2 in \cite{bre01} the space $\Wem pMN$ is always path connected when $1<p<2$, while there may exists distinct $p$-quasihomotopy classes (see the example above).

\subsubsection*{Outline}
The proof of Theorem \ref{main} is based on approximating a given Sobolev map with suitable mollified maps and showing the convergence is quasiuniform (Theorem \ref{molli}). The second section is devoted to mollification and the use of singular integrals to accomplish this. 
	
Section 3 deals with the aspherical case. For nonpositively curved targets Proposition \ref{pqpath} follows directly from Theorem 1.1 and Proposition 1.5 in \cite{teri1} but the more general case of aspherical targets requires somewhat different arguments and the use of Theorem \ref{molli}. Theorem \ref{main2} is an immediate consequence of Theorem \ref{main3}, presented in this Section. 

The last Section is devoted to proving that $\Wem pM{S^k}$ is $p$-quasiconnected, i.e. any two maps in $\Wem pM{S^k}$ are $p$-quasihomotopic,  when $p\le k$ (Proposition \ref{mn}). Some of the auxiliary results  (e.g. Proposition \ref{cap}) may be interesting in themselves. Proposition \ref{mn} serves as an example showing that sometimes -- though not in general -- path homotopy - and $p$-quasihomotopyclasses coincide.

The paper is closed by remarking that $\Wem p{B^{k+1}}{S^k}$, while path connected when $p<k+1$, is not $p$-quasiconnected for $k<p<k+1$.

\section{Critical exponent case}
The proof strategy of Theorem \ref{main} utilizes Brian White's result. 

\begin{theorem}[\cite{whi86}, Theorem 0 and \cite{bre03}, Theorem 2]\label{white}
	Two Lipschitz maps in \newline $\Wem nMN$ are path homotopic if and only if they are homotopic. Moreover for each $u\in \Wem nMN$ there is a number $\varepsilon >0$ so that if $\|u-v\|_{1,n}<\varepsilon$ then $u$ and $v$ are path homotopic. 
\end{theorem}

Coupled with the fact, due to Schoen-Uhlenbeck \cite{uhl83}, that $Lip(M;N)$ is dense in $\Wem nMN$ the question, whether path homotopy implies $n$-quasihomotopy, is reduced to the following statement. \emph{For every $u\in \Wem nMN$ and $\varepsilon>0$ there is a Lipschitz map $u_\varepsilon$ with $\|u-u_\varepsilon\|_{1,n}<\varepsilon$ such that $u_\varepsilon$ is $n$-quasihomotopic to $u$.}

We will construct such functions by means of mollifying the original function.

\subsection{Mollifiers} Suppose $\psi:[0,\infty)\to [0,1]$ is a Lipschitz cut-off function with $\spt\psi\subset [0,1)$. 
 Given $r>0$ define $\psi_r:M\to \R$ by $$\psi_r(p)=\int_M\psi\left(\frac{|p-z|}{r}\right)\ud z.$$ 
\begin{definition}
	Given $u\in L^p(M;\R^\nu)$ and $r>0$ set \[ \psi_r\ast u(p)=\frac{1}{\psi_r(p)}\int_M\psi\left(\frac{|p-z|}{r}\right)u(z)\ud z, \ p\in M \] 
\end{definition}
\begin{lemma}\label{lipmol}
	For each $u\in L^1_{loc}(M;\R^\nu)$ and  $r>0$ the map $\psi_r\ast u:M\to \R^\nu$ is Lipschitz continuous. Moreover $$\psi_r\ast u (x)\le C r\M |u| (x)$$ for almost every $x\in M$, with $C$, depending only on $\psi,\ M$ and $\nu$. 
\end{lemma}
\begin{proof}
For $g\in L^1_{loc}(M)$ and arbitrary $x,y\in M$ we have 
\begin{align}
&\left|\int_M\psi\left(\frac{|x-z|}{r}\right)g(z)\ud z-\int_M\psi\left(\frac{|y-z|}{r}\right)g(z)\ud z\right|\nonumber \\
\le & \Lip(\psi) \int_{B(x,r+d(x,y))}\left|\frac{|x-z|-|y-z|}{r}\right||g(z)|\ud z\nonumber\\
\le & \Lip(\psi)\frac{d(x,y)}{r}\int_{B(x,r+d(x,y)) }|g|\ud z.\label{apu}
\end{align}
The lipschitz continuity of $\psi_r\ast u$ follows from this by expressing the difference $\psi_r\ast u(x)-\psi_r\ast u(y)$ ,where $d(x,y)<r$, as \begin{align*}
\frac{\psi_r(y)-\psi_r(x)}{\psi_r(x)\psi_r(y)}&\int_M\psi\left(\frac{|x-z|}{r}\right)u(z)\ud z \\
+\frac{1}{\psi_r(y)}&\left[\int_M\psi\left(\frac{|x-z|}{r}\right)u(z)\ud z-\int_M\psi\left(\frac{|y-z|}{r}\right)u(z)\ud z \right]
\end{align*}
and applying (\ref{apu}) and the doubling property of the measure.

The estimate in the claim follows by a standard decomposition of the integral into annular regions, see \cite{HKST07, hei01}.
\end{proof}

\begin{lemma}\label{0}(Schoen-Uhlenbeck)
	Let $u\in \Wem pMN$. For $r>0$ we have \[ \dist(N,\varphi_r\ast u(p))\lesssim \left(\int_{B_r(p)}|Du|^n\ud z\right)^{1/n} \] for all $p\in M$. Consequently for each $u\in \Wem nMN$ there is $r_0>0$ so that $$\sup_{p\in M}\dist(N,\varphi_r\ast u(p))<\varepsilon_0$$ whenever $r<r_0$.
\end{lemma}
\begin{proof}
	Let $p\in M$. For a.e. $z\in B_r(p)$ $$\dist(N,\varphi_r\ast u(p))\le \|u(z)-\varphi_r\ast u(p)\|.$$ Taking an average integral over $B_r(p)$ we obtain $$\dist(N,\varphi_r\ast u(p))\le \dashint_{B_r(p)}\|u(z)-\varphi_r\ast u(p)\|\ud z. $$ By the $(1,n)$-Poincare inequality (which every manifold of dimension $n$ supports)
	\begin{align*}
	\dashint_{B_r(p)}\|u(z)-\varphi_r\ast u(p)\|\ud z&\le \frac{1}{\varphi_r(p)}\int_{B_r(p)}\dashint_{B_r(p)}\varphi\left(\frac{|p-w|}{r}\right)\|u(z)-u(w)\|\ud z\ud w\\
	&\lesssim \dashint_{B_r(p)}\dashint_{B_r(p)}\|u(z)-u(w)\|\ud z\ud w\\
	&\lesssim r\left(\dashint_{B_r(p)}|Du|^n\ud z\right)^{1/n}\simeq \left(\int_{B_r(p)}|Du|^n\ud z\right)^{1/n}.
	\end{align*}
	The implied constants in the estimates depend only on the data of $M$ and on $N$. The second assertion follows directly from the absolute continuity of the measure $|Du|\ud z$.
\end{proof}

\subsection{Singular integrals} Let us set some notation. Let $\varphi:\R\to \R$ be a smooth cutoff function and define the \emph{ kernel } $k_r:(0,\infty)\to \R$, $$k_r(t)=\frac{\varphi(t/r)}{t^{n-1}}.$$ We abuse notation by writing $$k_r(p,q)=k_r(|p-q|),\ p,q\in M$$ and finally, given $g\in L^p(M)$ ($1<p<\infty$), we define the convolution $$k_r\ast g (x)=\int_Mk_r(x,z)g(z)\ud z, \ x\in M.$$ By the compactness of $M$ there exists $r_1$ so that $$\exp_x:B^n(r_1)\to B(x,r_1) $$ is a 2-bilipschitz diffeomorphism for all $x\in M$. Thus, when $r<r_1$ we may use a change of variables given by the exponential map and write the integral above $$k_r\ast g (x)=\int_{B^n(r)}k_r(|\xi|) g(\exp_x\xi)J\exp_x\xi \ud\xi.$$

\begin{lemma}
	Let $1<p<\infty$. Given $g\in L^p(M)$ the function $k_r\ast g$ has distributional gradient 
	\[\nabla _x(k_r\ast g)v= -PV \int_M k_r'(|x-z|)\langle \nabla_xd_z,v\rangle g(z)\ud z,\ v\in T_xM. \]
\end{lemma}
\begin{proof}
	We refer to \cite{see59, cal56} for the existence and basic properties of singular integrals on manifolds (see in particular Chapter IV in \cite{koh57} and the example in \cite[D]{see59}.)
	
	The distributional derivative is determined by the condition \[ \int_M\langle \nabla(k_r\ast g),V\rangle \ud x=-\int_M(k_r\ast g)div V\ud x \] for all smooth vector fields $V$ on $M$. We may write 
	\begin{align}
	&\int_M PV \int_M k'_r(|x-z|)\langle \nabla_xd_z,V_x\rangle g(z)\ud z\ud x\nonumber \\
	=&\lim_{\delta\to 0}\int_M \int_{M\setminus B_\delta(x)} k'_r(|x-z|)\langle \nabla_xd_z,V_x\rangle g(z)\ud z\ud x\nonumber \\
	=&\lim_{\delta\to 0}\int_Mg(z)\int_{M\setminus B_\delta(z)}k'_r(|x-z|)\langle \nabla_xd_z,V_x\rangle \ud x\ud z.\label{pv}
	\end{align}
	
	Note that when $x\ne z$ the vector $\nabla_xd_z$ is the unit vector normal to $\partial B_\delta (z)$ at $x$. Thus $-\nabla_xd_z$ is the unit normal to $\partial(M\setminus B_\delta(z))$ at $x$. The divergence theorem gives
	\begin{align}
	&\int_{M\setminus B_\delta(z)}k'_r(|x-z|)\langle \nabla_xd_z,V_x\rangle \ud x \nonumber\\
	=&-\int_{M\setminus B_\delta(z)}k_r(|x-z|)div V_x\ud x+ \int_{\partial B_\delta(z)}k_r(|y-z|)\langle \nabla_yd_z,V_y\rangle \ud\sigma(y)\nonumber \\
	=&-\int_{M\setminus B_\delta(z)}k_r(|x-z|)div V_x\ud x+O(\delta)\label{o}.
	\end{align}
	The second term is $O(\delta)$ since it may be estimated using again the divergence theorem: \[ \left|\int_{\partial B_\delta(z)}k_r(|y-z|)\langle \nabla_yd_z,V_y\rangle \ud\sigma(y)\right|=\left|k_r(\delta)\int_{B_\delta(z)}div V_y\ud y \right| \lesssim \delta^{1-n} \delta^n.\]
	
	Plugging (\ref{o}) in (\ref{pv}) we obtain
	\begin{align}
	&\int_M PV \int_M k'_r(|x-z|)\langle \nabla_xd_z,V_x\rangle g(z)\ud z\ud x\nonumber\\
	=& -\lim_{\delta\to 0}\int_M g(z)\int_{M\setminus B_\delta(z)}k_r(|x-z|)div V_x\ud x \ud z + \lim_{\delta\to 0}\int_MO(\delta)\ud z\nonumber\\
	=&- \lim_{\delta\to 0}\int_M \int_{M\setminus B_\delta(x)}k_r(|x-z|)g(z)div V_x\ud z \ud x = -\int_M(k_r\ast g)divV\ud x.
	\end{align}
	Thus we are done.
\end{proof}
\begin{lemma}\label{unibound}
	The operators $g\mapsto k_r\ast g$ ($r>0$) are uniformly bounded  $$L^p(M)\to \W pM, $$ i.e. 
	\begin{equation}\label{bound}
	\int_M|k_r\ast g(x)|^p\ud x+\int_M|\nabla_x(k_r\ast g)|^p\ud x\le C\int_M|g|^p\ud x,
	\end{equation} $g\in L^p(M)$, for all $0<r<r_1$.
\end{lemma}
\begin{proof} For a.e. $x\in M$ we have, $v\in T_xM$ and $r>0$
	\begin{align*}
	&|\nabla_x(k_r\ast g)v|=\left|PV \int_Mk'_r(|x-z|)\langle\nabla_xd_z,v\rangle g(z)\ud z\right| \\
	\le & |v|\int_M\frac{|\varphi'(|x-z|/r)}{r|x-z|^{n-1}}|g(z)|\ud z+ \left|PV \int_M\frac{\varphi(|x-z|/r)}{|x-z|^n}\langle \nabla_xd_z,v\rangle g(z)\ud z\right|.
	\end{align*}
	Using this and the estimate in Lemma \ref{lipmol} we obtain the estimate
	\begin{align}
	|k_r\ast g(x)|^p+|\nabla_x(k_r\ast g)|^p \le & C(r^p+1)\M g(x)^p\nonumber \\
	&+\sup_{|v|=1} \left|PV \int_M\frac{\varphi(|x-z|/r)}{|x-z|^n}\langle \nabla_xd_z,v\rangle g(z)\ud z\right|^p\label{yet}
	\end{align}
	In light of (\ref{yet}) it suffices to demonstrate the (uniform) boundedness of $T_j:L^p(M)\to L^p(M)$ given by $$ T_jg(x)=  PV \int_M\frac{\varphi(|x-z|/r)}{|x-z|^n}\langle \nabla_xd_z,\partial_j\rangle g(z)\ud z$$ for each $j=1,\ldots,\dim M$, when $r<r_1$.
	\begin{sublemma}
		The operator $T_j:L^p(M)\to L^p(M)$ is bounded with norm independent of $r\in (0,r_1)$. 
		\begin{proof}[Proof of sublemma]
			Since $r<r_1$ and the integrand in $T_j$ vanishes outside $B(x,r)$ which is bilipschitz diffeomorphic to $B^n(r)$ through the exponential map $\exp_x:B^n(r)\to B(x,r)$, the operator $T_j$ may be written \[T_jg(x)=PV\int_{B^n(r)}\varphi(|\xi|/r)\frac{\xi_j}{|\xi|^{n+1}}g(\exp_x\xi)J\exp_x(\xi)\ud \xi. \]
			By Definition 4 in \cite[B]{see59} it is sufficient to prove the boundedness, uniformly in $r$, for the Euclidean operator \newline $\widetilde T_j:L^p(\R^n)\to L^p(\R^n)$ given by the same kernel: \[ \widetilde T_jh(x)=PV\int_{\R^n}\varphi(|\xi|/r)\frac{\xi_j}{|\xi|^{n+1}}h(x-\xi)\ud\xi. \]
			
			By Theorem 5.4.1 in \cite{gra14} (cf. Chapter 5, Theorem 5.1 in \cite{duo01}) this is implied by the following two conditions. Denote \[\displaystyle K_r(y)=\varphi(|y|/r)\frac{y_j}{|y|^{n+1}}.\]
			\begin{itemize}
				\item[(1)] $\|\widehat{K_r}\|_{\infty}\le A$, and
				\item[(2)] $|\nabla K_r(y)|\le \frac{B}{|y|^{n+1}}$.
			\end{itemize}
			A change of variables implies $\widehat{K_r}(\xi)=\widehat{K_1}(r\xi)$ so that $$\|\widehat{K_r}\|_{\infty}\le \|\widehat{K_1}\|_{\infty}:=A.$$ We may estimate $$|\nabla K_r(y)|\le \chi_{B^n(r)}(y)\left[\frac{\|\varphi'\|_\infty}{r|y|^n}+\|\varphi\|_\infty|\nabla(y_j/|y|^{n+1})|\right]\le \frac{C(n,\varphi)}{|y|^{n+1}}.$$ Consequently both (1) and (2) are satisfied with constants independent of $r$. This completes the proof of the sublemma.
		\end{proof}
	\end{sublemma}
	Having a bound $\|T_j\|_{L^p(M)\to L^p(M)}\le C$ where $C$ is independent of $r$ we obtain the estimate (\ref{bound}) with constant $C$ independent of $r$. This proves Lemma \ref{unibound}.
\end{proof}

\subsection{The proof of Theorem \ref{main}} Using Lemma \ref{0} we define a net of approximating maps with values in $N$.
\begin{definition}\label{mollidef}
	Let $u\in \Wem nMN$, and let $r_0$ be the constant in Lemma \ref{0}. For $0<r\le r_0$ set \[ u_r(p)=\pi(\varphi_r\ast u(p)),\ p\in M. \] Additionally, we set $$u_0=u.$$
\end{definition}
For each $r>0$ the maps $u_r:M\to N$ are clearly Lipschitz. \emph{The resulting map $M\times [0,r_0]\ni (p,t)\mapsto u_t(p)$ is a key component in the proof of Theorem \ref{main}.}

\begin{lemma}\label{easy}
	Let $r>0$. Then $u_s\to u_r$ uniformly as $s\to r$.
\end{lemma}
\begin{proposition}\label{hard}
	The maps $u_r$ converge $n$-quasiuniformly to $u$, i.e. for each $\varepsilon >0$ there exists an open set $U$ with $\Cp_n(U)<\varepsilon$ such that $(u_r)|_{M\setminus U}\to u|_{M\setminus U}$ uniformly as $r\to 0$.
\end{proposition}

\begin{proof}[Proof of Lemma \ref{easy}]
We will estimate the difference $\|u_r(p)-u_s(p)\|$ by splitting it into two parts. Let $b$ be any vector in $\R^\nu$. We will later choose it appropriately.
\begin{align}
\|u_r(p)-u_s(p)\|\le & \|\varphi_r\ast u(p)-\varphi_s\ast u(p)\|=\|\varphi_r\ast [u-b](p)-\varphi_s\ast [u-b](p)\|\nonumber\\
\label{1}\le & \left|\frac{1}{\varphi_r(p)}-\frac{1}{\varphi_s(p)} \right|\int_{B_{r}(p)}\varphi\left(\frac{|p-z|}{r}\right)\|u(z)-b\|\ud z\\
\label{2} &+ \frac{1}{\varphi_s(p)}\int_{B_{s\vee r}(p)}\left|\varphi\left(\frac{|p-z|}{r}\right)-\varphi\left(\frac{|p-z|}{s}\right) \right|\|u(z)-b\|\ud z
\end{align}
Let us estimate the two terms (\ref{1}) and (\ref{2}) separately, starting with the latter. Throughout we assume that $|r-s|<r$, which implies that $\varphi_{r\vee s}(p)\lesssim \varphi_s(p)$ with constant depending only on $M$.
\begin{align*}
(\ref{2})&\le \frac{1}{\varphi_s(p)}\int_{B_{s\vee r}(p)}\Lip(\varphi)|p-z|\left|\frac{1}{r}-\frac{1}{s} \right|\|u(z)-b\|\ud z\\
&\lesssim |r/s-1|\vee |s/r-1|\dashint_{B_{r\vee s}(p)}\|u-b\|\ud z.
\end{align*}
A similar computation yields the same bound for (\ref{1}). Thus we arrive at \[ \|u_r(p)-u_s(p)\|\lesssim  |r/s-1|\vee |s/r-1|\dashint_{B_{r\vee s}(p)}\|u-b\|\ud z. \]
Now we choose $b=u_{B_{s\vee r}(p)}$ and use the $(1,n)$-Poincare inequality to estimate $$\dashint_{B_{r\vee s}(p)}\|u-b\|\ud z\lesssim \left(\int_{B_{s\vee r}(p)}|Du|^n\ud z\right)^{1/n}\lesssim \|Du\|_{L^n(M)}. $$ Combining these we arrive at
\begin{align*}
\|u_r(p)-u_s(p)\|\lesssim (|r/s-1|\vee |s/r-1|)\|Du\|_{L^n(M)}
\end{align*}
for all $p$. Thus $u_s\to u_r$ uniformly as $s\to r$, as long as $r\ne 0$.
\end{proof}

\bigskip\noindent Proposition \ref{hard} requires more work. We begin by estimating the difference of $u$ and $u_r$ by an expression which we study in more detail
\begin{lemma}\label{est}
	Let $u\in \Wem pMN$. For $p$-q.e. $x\in M$ we have \[ \|u(x)-u_r(x)\|\le \int_M\frac{\varphi(|z-x|/r)}{|z-x|^{n-1}}|Du|(z)\ud z. \]
\end{lemma}

\begin{proof}
		The proof is similar to \cite[p. 28, (4.5)]{hei01}.
\end{proof}

\begin{lemma}\label{quasiuni}
	Let $1<p<\infty$ and let  $(f_k)\subset \Ne pM$ be a bounded secuence with $0\le f_{k+1}\le f_k$ pointwise and $\|f_k\|_{L^p}\to 0$ as $k\to \infty$. Then $f_k\to 0$ $p$-quasiuniformly.
\end{lemma}

\begin{proof}
	Since $\Ne pM$ is reflexive we may pass to a subsequence converging weakly to 0, and by the Mazur lemma a sequence of convex combinations converges to 0 in norm. Passing to another subsequence if needed, we may assume that the sequence of convex combinations, $$h_m=\lambda_1^mf_{k_1}+\cdots +\lambda_{N_m}^mf_{k_{N_m}},\ (k_1<\cdots<k_{N_m}),$$ converges to zero $p$-quasiuniformly. The monotonicity now imples $$0\le f_{k_{N_m}}\le h_m$$ so that a subsequence of $(f_k)$ converges $p$-quasiuniformly to zero. Since the sequence is pointwise nonincreasing the whole sequence converges to zero $p$-quasiuniformly.
\end{proof}

\bigskip\noindent These auxiliary results yield Proposition \ref{hard}.
\begin{proof}[Proof of Proposition \ref{hard}]
	By Lemma \ref{est} we have \[ \|u(x)-u_r(x)\|\lesssim k_r\ast |Du| (x)  \]for $p$-quasievery $x\in M$. Choosing $\varphi$ nonincreasing we get that $$k_r\ast |Du|\le k_s\ast |Du|$$ pointwise whenever $r<s$, and further, \[ k_r\ast |Du|\stackrel{L^n}{\longrightarrow} 0 \]as $r\to 0$. By lemma \ref{unibound} the functions $k_r\ast |Du|$ have uniformly bounded $W^{1,n}$-norms (in $r$) so by Lemma \ref{quasiuni} we have that $k_r\ast |Du|\to 0$ $n$-quasiuniformly. Consequently $u_r\to u$ $n$-quasiuniformly.
\end{proof}

\begin{theorem}\label{molli}
	Let $u\in \Wem nMN$. The map $M\times [0,r_0]\to N$ given by \[ (p,r)\mapsto u_r(p) \] in \ref{mollidef} defines an $n$-quasihomotopy $u\simeq u_{r_0}$.
\end{theorem}
\begin{proof}
	Denote $H(p,r)=u_r(p)$ and suppose $\varepsilon>0$ is given. Let $U$ be the open set satisfying the claim of Proposition \ref{hard}. We claim that $H|_{M\setminus U\times [0,r_0]}$ is continuous. For this it suffices to show that $(u_s)|_{M\setminus U}\to (u_r)|_{M\setminus U}$ uniformly as $s\to r$. This, however, follows immediately from \ref{easy} and \ref{hard}.
\end{proof}

\bigskip\noindent We close this Section with the proof of Theorem \ref{main}.
\begin{proof}[Proof of Theorem \ref{main}]
	Suppose $u,v\in \Wem nMN$ are path homotopic. For small enough $\varepsilon$ we have, by Theorems \ref{molli} and \ref{white}, that $u_\varepsilon$ is both $n$-quasihomotopic and path homotopic to $u$. The same holds for $v$ and $v_\varepsilon$. 
	
	It follows that $u_\varepsilon$ and $v_\varepsilon$ are path homotopic and since they are Lipschitz, homotopic (Theorem \ref{white}). 
	
	Thus $u_\varepsilon$ and $v_\varepsilon$ are $n$-quasihomotopic. Consequently $u$ and $v$ are $n$-qua\-si\-ho\-mo\-to\-pic.
\end{proof}

\section{Aspherical targets}

A topological space $X$ is called \emph{aspherical} if $\pi_i(X)=0$ for every $i>1$. It is well known that for smooth Riemannian manifolds the vanishing of higher homotopy groups is equivalent to having contractible universal cover. In particular manifolds with nonpositive sectional curvature are aspherical. The equivalence stated in Theorem \ref{main2} can be seen as a Sobolev version of Whiteheads theorem \cite{hat02}.

Before turning our attention to Theorem \ref{main2} let us present a proof of Theorem \ref{pqpath}.
\begin{proof}[Proof of Theorem \ref{pqpath}]
	Suppose $N$ is aspherical and let $f,g\in \Wem pMN$ be $p$-quasihomotopic. We devide the proof into three cases:
	\begin{itemize}
		\item[(1) $p<n$:] By Theorem 1.4 in \cite{teri1} $f$ and $g$ are path homotopic.
		\item[(2) $p>n$:] In this case path homotopy and $p$-quasihomotopy coincide, see the discussion in the introduction.
		\item[(3) $p=n$:] This is the only case that requires some work. By Theorem \ref{molli} $f,g$ are $n$-quasihomotopic to Lipschitz maps $f_0,g_0$ so we may assume that $f$ and $g$ are themselves Lipschitz. Since $N$ is aspherical it is path representable \cite[Proposition 3.4]{teri2} and thus by \cite[Theorem 1.2]{teri2} $(f,g)\in \Nem nMN\cap \Lip(M;N)$ has a lift $h\in \Nem nM{\widehat{N}_{diag}}$ where $\widehat{N}_{diag}$ is the diagonal cover of $N$ (see \cite[Subsection 2.4]{teri2}). Since $g_h=g_{(f,g)}\le \LIP(f)+\LIP(g)$ almost everywhere (Lemma 4.3 in \cite{teri2}) it follows that $h$ is in fact Lipschitz. Thus the continuous map $(f,g):M\to N\times N$ admits a (continuous) lift \newline $h:M\to \widehat{N}_{diag}$. By Proposition 3.2 in \cite{teri2} $f$ and $g$ are homotopic, hence path homotopic in $\Wem nMN$.
	\end{itemize}
\end{proof}

When $p\ge 2$, a Sobolev map $f\in \Wem pMN$ induces a homorphism\newline $u_\ast:\pi(M,x_0)\to \pi(N,f(x_0))$ \cite{sch79} (see also \cite{whi88, nak93}). For almost every $x_0\in M$ an induced homomorphism satisfies, for all $[\gamma]\in \pi(M,x_0)$:
\begin{itemize}
\item $u_\ast[\gamma]=[u\circ\gamma]$ if $\gamma$ is such that $u\circ\gamma$ is continuous
\item $u_\ast[\gamma]=[u\circ \gamma']$ for some $\gamma'\sim \gamma$. 
\end{itemize}
It is known that no such induced homomorphism need exist for a Sobolev map $f\in \Wem pMN$ when $1<p<2$.

To connect induced homomorphisms to $p$-quasihomotopies we recall the notion of a \emph{fundamental system of loops} from \cite{teri2}.

\bigskip\noindent Given a $p$-quasicontinuous representative $u\in \Wem pMN$, an upper gradient $g\in L^p(M)$ and an exceptional path family $\Gamma_0$ of curves in $M$, such that $g$ is an upper gradient of $u$ along any curve $\gamma\notin \Gamma_0$, and a \emph{basepoint} $x_0\in M$ with $\M g^g(x_0)<\infty$, the collection of loops \[ \F_{x_0}(g,\Gamma_0)=\{\alpha\beta^{-1}: \Gamma_{x_0x}\setminus\Gamma_0,\ \M g^p(x)<\infty \} \] is called the fundamental system of loops.

Recall the definition of $\spt_p\Gamma_0$ of a negligible path family: $$\spt_p\Gamma_0=\bigcap\{\M \rho^p=\infty\}$$ where the intersection is taken over all admissible metrics $\rho\in L^p(M)$ for which $$\int_\gamma\rho=\infty\textrm{ for all }\gamma\in \Gamma_0.$$

\begin{lemma}\label{disse}
	There is a constant $C$ with the following property. If $\Gamma_0$ is a path family and $g\in L^p(M)$ a nonnegative Borel function with $$\int_{\gamma}g=\infty,\ \gamma\in \Gamma_0, $$  then for any $x,y\notin \{\M g^p=\infty\}$ there exists a curve $\gamma\notin \Gamma_0$ joining $x$ and $y$ with $$\ell(\gamma)\le Cd(x,y).$$
\end{lemma}
\begin{proof}
By Lemma 4.5 in \cite{teri2} and Theorem 2 (4) in \cite{kei03} we have \[ d(x,y)^{1-p}\le C\Mod_p(\Gamma_{xy}\setminus \Gamma_g; \mu_{xy}), \] where \[ \mu_{xy}(A)=\int_A\left[\frac{d(x,z)}{\mu(B(x,d(x,z)))}+\frac{d(y,z)}{\mu(B(y,d(y,z)))} \right]\ud\mu(z), \ A\subset X. \] In particular $\Gamma_{xy}\setminus \Gamma_g$ is nonempty. Note that $\Gamma_0\subset \Gamma_g$.

If $\ell(\gamma)\ge Dd(x,y)$ for all $\gamma\in \Gamma_{xy}\setminus \Gamma_g\subset \Gamma_{xy}\setminus \Gamma_0$ then $\rho=1/(Dd(x,y))$ is admissible  for $\Gamma_{xy}\setminus \Gamma_g$ and thus \[ \Mod_p(\Gamma_{xy}\setminus \Gamma_g; \mu_{xy})\le CD^{-p}d(x,y)^{1-p}. \]Combining the two inequalitites yields the required bound on $D$.
\end{proof}
	
\begin{lemma}\label{loop}
	Let $p\ge 2$, and $u\in \Wem{p}{M}{N}$ be a quasicontinuous representative. Given an upper gradient $g$ of $u$, a path family $\Gamma_0$ of zero $p$-modulus, and a point $x_0\notin \spt_p\Gamma_0$ with $\M g^p(x_0)<\infty$,  we have \[ u_\ast\pi(M,x_0)=u_\sharp \F_{x_0}(g,\Gamma_0). \]
\end{lemma}
\begin{proof}
Let $u\in \Wem pMN$ and let $g,\Gamma_0$ be as in the claim. Set $$E=\{x_0: \M g^p(x_0)=\infty \}\cup \spt\Gamma_0$$ and choose and arbitrary point $x_0\notin E$. For any $\gamma\in \F_{x_0}(g,\Gamma_0)$ clearly $[u\circ \gamma]\in u_\ast\pi(M,x_0)$. Thus we only need to prove the other inclusion.

To this end, fix a loop $\gamma$ based on $x_0$. Take a tubular neighbourhood $T$ of $\gamma$ so that any loop in $T$ is homotopic with $\gamma$. Take a finite chain of open balls $x_0\in B_0,B_1,\ldots,B_k$ of radii $r>0$ such that $2C\overline B_j\subset T$, and $B_j\cap B_{j+1}\ne \varnothing$, where $C$ is the constant in Lemma \ref{disse}. Since $|E|=0$ there exists, for each $j$, points $y_j\in (B_j \cap B_{j+1})\setminus E$ (with the convention that $y_0=x_0$ and $y_k\in (B_0\cap B_k)\setminus E$.) 

By Lemma \ref{disse} there exists a curve $\gamma_j\notin \Gamma_0$ joining $y_j$ and $y_{j+1}$ with $\ell(\gamma_j)\le Cd(y_j,y_{j+1})$ (here $y_{k+1}=x_0$). Hence $|\gamma_j|\subset T$. The loop $\gamma'=\gamma_0\cdots\gamma_{k+1}$ belongs to $\F_{x_0}(g,\Gamma_0)$ and is contained in $T$, and therefore homotopic with $\gamma$.

It follows that $[u\circ \gamma']=u_\ast[\gamma']=u_\ast[\gamma]$ and since $\gamma$ was arbitrary we obtain $u_\ast\pi(M,x_0)\le u_\sharp\F_{x_0}(g,\Gamma_0)$. The proof is complete.
\end{proof}

\begin{lemma}\label{pconj}
	Let $p\ge 2$. Two maps, $u,v\in \Wem pMN$, are $p$-quasihomotopic if and only if $ u_\sharp\pi(M)$ and $v_\sharp\pi(M)$ are conjugated subgroups of $\pi(N)$.
\end{lemma}
\begin{proof}
By \cite[Theorem 1.2 and 1,3]{teri2} the maps $u,v$ are $p$-quasihomotopic if and only if
\begin{equation}\label{eq z}
(u,v)_\sharp\F_{x_0}(g,\Gamma_0)\le p_\ast\pi(\widehat{N}_{diag},[\alpha])
\end{equation} for some $[\alpha]\in p^{-1}(u(x_0),v(x_0))$, and some $x_0\in M$. Here $(p,\widehat{N}_{diag})$ is the \emph{diagonal cover} of $N$ which consists of homotopy classes of all paths in $N$ (see \cite{teri2} for the precise construction). A modification of the proof of \cite[Lemma 2.18]{teri2} yields \[ p_\ast\pi(\widehat{N}_{diag},[\alpha])=\{([\gamma],[\alpha^{-1}\gamma \alpha]): [\gamma]\in \pi(N,u(x_0)) \}\le \pi(N,u(x_0))\times\pi(N,v(x_0)). \] On the other hand by Lemma \ref{loop} \[ (u,v)_\sharp\F_{x_0}(g,\Gamma_0)= (u,v)_\ast\pi(M,x_0)=\{ (u_\ast[\gamma],v_\ast[\gamma]): [\gamma]\in \pi(M,x_0)\}. \] By these two identities (\ref{eq z}) is equivalent to \[ u_\ast[\gamma]=[\alpha]^{-1}v_\ast[\gamma][\alpha]  \] for all $[\gamma]\in \pi(M,x_0)$. Hence we are done.
\end{proof}

\begin{lemma}\label{pathconj}
	If $u,v\in \Wem pMN$ are path homotopic ($p\ge 2$) then for almost every $x_0\in M$ $u_\ast\pi(M, x_0)$ and $v_\ast\pi(M, x_0)$ are conjugated. 
\end{lemma}
\begin{proof}
	Suppose first that $p<n$. Then by \cite[Theorem 1.1]{han03} $u$ and $v$ are $[p-1]$-homotopic and, since $p\ge 2$, in particular $1$-homotopic. Fix a $1$-skeleton $K$ of $M$ containing a point $x_0\in \{\M (|Du|^p+|Dv|^p)<\infty \}$, and such that $u|_{K}$ and $v|_{K}$ are (continuous and) homotopic by a homotopy $h:K\times [0,1]\to N$.
	
	To prove that the image subgroups of the homomorphisms are conjugated, take a loop $\gamma$ with basepoint $x_0$. By \cite[Section 4.1, Theorem 4.8]{hat02} $\gamma$ is homotopic to a loop $\gamma'$ which lies in $K$. Thus the image loops $u\circ\gamma'$ and $v\circ\gamma'$ are conjugated by $$H(s,t)=h(\gamma(s),t),\ t,s\in [0,1]^2.$$ Denoting by $\alpha$ the path $t\mapsto h(x_0,t)$ we thus have $$[u\circ\gamma']=[\alpha^{-1}(v\circ\gamma')\alpha].$$ Consequently $$u_\ast([\gamma])=u_\ast([\gamma'])=(v_\ast([\gamma']))^{[\alpha]}=(v_\ast([\gamma]))^{[\alpha]}, \ [\gamma]\in \pi(M,x_0).$$ This proves the claim in the case $p<n$.
	
	In case $p\ge n$ it follows from Theorem \ref{main} and Theorem \ref{homotopy} that $u$ and $v$ are $p$-quasihomotopic. The claim now follows from Lemma \ref{pconj} above.	
\end{proof}
\noindent Combining Proposition \ref{pqpath} and Lemmata \ref{pconj} and \ref{pathconj} we obtain the following theorem, which directly implies Theorem \ref{main2}.

\begin{theorem}\label{main3}
	Let $p\ge 2$, and $N$ aspherical. Then two maps $u,v\in \Wem pMN$ are path homotopic if and only if the subgroups $u_\ast\pi(M)$ and $v_\ast\pi(M)$ are conjugated.
\end{theorem}
\begin{proof}
	Suppose $u,v$ are path homotopic. Then Lemma \ref{pathconj} implies the claim. If, conversely, $u_\ast\pi(M)$ and $v_\ast\pi(M)$ are conjugated, Lemma \ref{pconj} implies that $u$ and $v$ are $p$-quasihomotopic. By Proposition \ref{pqpath} $u$ and $v$ are path homotopic.
\end{proof}

\section{Quasiconnectedness of $\Wem pM{S^k}$}

In this section the following result is proven.
	\begin{proposition}\label{mn}
		Suppose $M$ is a smooth compact riemannian manifold, possibly with boundary, and $1<p\le k$. Then $u\in \Wem pM{S^k}$ is $p$-quasiconnected, i.e. every map is $p$-quasihomotopic to a constant.
	\end{proposition}
	We single out the following corollary.
	\begin{corollary}\label{sn}
		Suppose $2 \le k$ and $1<p\le k$. Then any two maps in $\Wem p{S^k}{S^k}$ are $p$-quasihomotopic.
	\end{corollary}
\bigskip\noindent The proof of Theorem \ref{mn} is based on the example given in \cite{bre03} after Theorem 3. We begin by observing that that in a suitable range of $p$'s points have small preimages under Sobolev maps. 

\begin{lemma}\label{cap}
	Let $f\in \Wem pMN$ be a $p$-quasicontinuous representative, $1<p\le \dim N$. Then for almost every $y\in N$ we have \[ \Cp_p(f^{-1}(y))=0. \] 
\begin{proof}
	For $y\in N$, consider the function $u_k\in \W pM$ given by $$ u_k(x)=w_k\circ f,$$ where $w_k:N\to \R$ is defined by
	\begin{align*}
	w_k(z)=\left\{
	\begin{array}{ll}
	1 &, z\in B(y,1/k^2)\\
	(\log k)^{-1}\log\left(\frac{1/k}{|z-y|}\right) &, z\in A(y,1/k^2,1/k)\\
	0 &, z\notin B(y,1/k)
	\end{array}
	\right.
	\end{align*} Then $u_k|_{f^{-1}(y)}\equiv 1$ $p$-quasieverywhere and therefore $$\Cp_p(f^{-1}(y))\le \liminf_{k\to\infty}\|u_k\|_{1,p}^p.$$ We have the pointwise estimates 
	\begin{align*}
	&0\le u_k(x)\le \chi_{B(y,1/k)}(f(x)),\\
	&|\nabla u_k|(x)\le |\nabla w_k|(f(x))|\nabla f|(x)\le (\log k)^{-1}\frac{\chi_{A(y,1/k^2,1/k)}(f(x))}{|f(x)-y|}|\nabla f|(x)
	\end{align*} 
	almost everywhere. Thus 
	\begin{align*}
	\Cp_p(f^{-1}(y))\le & \liminf_{k\to\infty}\left[\int_M\chi_{B(y,1/k)}\circ f \ud x\right. \\
	&\left. +(\log k)^{-p}\int_M\frac{\chi_{A(y,1/k^2,1/k)}(f(x))}{|f(x)-y|^p} |\nabla f|^p\ud x\right].
	\end{align*} 
	Integrating over $y\in N$ and using Fatou and Fubini we obtain 
	\begin{align}\label{cap1}
	&\int_N\Cp_p(f^{-1}(y))\ud y\\
	\le & \liminf_{k\to\infty}\int_M\int_N\left[\chi_{B(y,1/k)}(f(x))+(\log k)^{-p}|\nabla f|^p(x)\frac{\chi_{A(y,1/k^2,1/k)}(f(x))}{|f(x)-y|^p}\right]\ud y\ud x \nonumber
	\end{align}
	Since \[ \int_M\int_N\chi_{B(y,1/k)}(f(x))\ud y\ud x =\int_M\left(\int_N\chi_{B(f(x),1/k)}(y)\ud y\right)\ud x \le C/k^{\dim N}  \]
	inequality (\ref{cap1}) becomes 
	\begin{align}\label{cap2}
	&\int_N\Cp_p(f^{-1}(y))\ud y\nonumber \\
	\le & \liminf_{k\to\infty}\int_M\int_N(\log k)^{-p}|\nabla f|^p(x)\frac{\chi_{A(y,1/k^2,1/k)}(f(x))}{|f(x)-y|^p}\ud y\ud x.
	\end{align}
	The righthand integral in turn may be written as 
	\begin{align*}
	(\log k)^{-p}\int_M|\nabla f|^p(x)\left(\int_N\frac{\chi_{A(f(x),1/k^2,1/k)}(y)}{|f(x)-y|^p}\ud y\right)\ud x.
	\end{align*}
	For sufficiently large $k\ge 1$ one may estimate
	\begin{align*}
	\int_N\frac{\chi_{A(f(x),1/k^2,1/k)}(y)}{|f(x)-y|^p}\ud y \lesssim C\int_{\R^{\dim N}}\chi_{A(0,1/k^2,1/k)}(y)\frac{\ud y}{|y|^p}\simeq \int_{1/k^2}^{1/k} t^{\dim N-1-p}\ud t.
	\end{align*}
	Since $p\le \dim N$ we obtain \[\int_{1/k^2}^{1/k} t^{\dim N-1-p}\ud t\le \int_{1/k^2}^{1/k} t^{-1}\ud t=\log k. \]
	Plugging all these inequalities into (\ref{cap2}) we obtain
	\begin{align*}
	\int_N\Cp_p(f^{-1}(y))\ud y\le C \liminf_{k\to\infty}\int_M(\log k)^{1-p}|\nabla f|^p(x)\ud x=0,
	\end{align*}
	thus completing the proof.
\end{proof}
\end{lemma}

\begin{corollary}\label{capcor}
	Let $2\le k$ and $1<p\le k$. For a $p$-quasicontinuous representative $f\in \Wem p{M}{S^k}$ the following holds for almost every $y\in S^k$. \[ \lim_{r\to 0}\Cp_p(f^{-1} B(y,r))=0. \]
\end{corollary}
\begin{proof}
	Let $\varepsilon>0$ be arbitrary and let $U\subset M$ be open with $\Cp_p(U)<\varepsilon$ and $f|_{M\setminus U}$ continuous. We may estimate \[ \Cp_p(f^{-1} B(y,r))\le \Cp_p((f|_{M\setminus U})^{-1}( \overline B(y,r)))+\Cp_p(U). \] The sets $(f|_{M\setminus U})^{-1}( \overline B(y,r))$ are compact and decrease to $(f|_{M\setminus U})^{-1}(y)$ as $r>0$  decreases. By the monotonicity of capacity for compact sets therefore $$\limsup_{r\to 0}\Cp_p((f|_{M\setminus U})^{-1}( \overline B(y,r)))=\Cp_p((f|_{M\setminus U})^{-1}(y)).$$ The latter quantity is zero for almost every $y\in S^k$ by Lemma \ref{cap} above. Thus we obtain \[ \Cp_p(f^{-1} B(y,r))\le 0+\Cp_p(U)<\varepsilon. \] Since $\varepsilon >0$ was arbitrary the claim follows.
\end{proof}

\begin{proof}[Proof of Proposition \ref{mn}]
	Suppose $f\in \Wem p{M}{S^k}$. Choose $y_0\in S^k$ so that the claim of Corollary \ref{capcor} holds for $y=y_0$. Define $h:S^k\times [0,\infty]\to S^k$ by
	\begin{align*}
	h(x,t)=\left\{\begin{array}{ll}
	\frac{x-ty_0}{|x-ty_0|},&\ 0\le t<\infty\\
	-y_0,&\ t=\infty
	\end{array}\right.
	\end{align*} 
	Note that $h|_{S^k\setminus\{x_0\}\times [0,\infty]}$ is continuous. We claim that $$H(x,t)=h(f(x),t),\ (x,t)\in M\times [0,\infty]$$ is a $p$-quasihomotopy $f\simeq -y_0$.
	
	Given $\varepsilon>0$ let $U$ be an open set with $\Cp_p(U)<\varepsilon/2$ and $f|_{M\setminus U}$ continuous. Further let $r>0$ be small enough so that $\Cp_p(f^{-1}B(y_0,r))<\varepsilon/2$. Set $E=U\cup [(f^{-1}B(x_0,r))\setminus U]$. Then $E$ is open, $\Cp_p(E)<\varepsilon$ and $H|_{M\setminus E\times [0,\infty]}$ is continuous, $$H(x,0)=\frac{f(x)}{|f(x)|}=f(x),\ H(x,\infty)=-x_0,\ x\in M\setminus E.$$
	
	
\end{proof}

\begin{remark}
	A similar procedure yields a continuous path in $\Wem p{S^n}{S^n}$ between $f$ and a constant map when $p<n$ (see \cite{bre03}), but not when $p=n$. Indeed, in the latter case it is not possible to connect every map to a constant path by a continuous path (\cite[Lemma 1'']{bre03}) and so we see that the converse of Theorem \ref{main} is not true.
\end{remark}

\bigskip\noindent In closing we remark that $\Wem p{B^{k+1}}{S^k}$, $k<p<k+1$ provides another example where path and $p$-quasihomotopy differ.

Consider the map $g:(0,1]\times S^k\to B^{k+1}$ given by \[ g(t,y)=ty. \]
This is a $p$-quasihomotopy equivalence ($p<k+1$) since the map $h(x)=(|x|,x/|x|)$ is $p$-quasicontinuous and $g\circ h=id_{B^{k+1}}$, $h\circ g=id_{(0,1]\times S^k}$ $p$-quasieverywhere. Thus, postcomposition with $g$ defines a continuous map \[G:\Wem p{B^{k+1}}{S^k}\to \Wem p{(0,1]\times S^k}{S^k},\ Gf=f\circ g, \] which preserves $p$-quasihomotopy classes and is bijective (the map $f\mapsto f\circ h$ is an inverse to $G$).

It is known (\cite{bre01}, Proposition 0.2) that $\Wem p{(0,1]\times S^k}{S^k}$ is path connected when $p<k+1$. However, when $k<p<k+1$, the Sobolev space $\Wem p{(0,1]\times S^k}{S^k}$ and consequently $\Wem p{B^{k+1}}{S^k}$ is not $p$-quasiconnected. (This easily seen by noting that the map $f(t,y)=y$, $(t,y)\in (0,1]\times S^k$, is not $p$-quasihomotopic to a constant map.)

\subsubsection*{Acknowledgements}

I would like to thank Pekka Pankka for reading the manuscript and making many valuable comments. I also thank Pawel Goldstein for useful discussions.

\bibliographystyle{plain}
\bibliography{abib}

\def\cprime{$'$} \def\cprime{$'$} \def\cprime{$'$} \def\cprime{$'$}
  \def\cprime{$'$}
\begin{thebibliography}{10}

\bibitem{bjo11}
Anders Bj{\"o}rn and Jana Bj{\"o}rn.
\newblock {\em Nonlinear potential theory on metric spaces}, volume~17 of {\em
  EMS Tracts in Mathematics}.
\newblock European Mathematical Society (EMS), Z\"urich, 2011.

\bibitem{bre03}
Ha{\"{\i}}m Brezis.
\newblock The fascinating homotopy structure of {S}obolev spaces.
\newblock {\em Atti Accad. Naz. Lincei Cl. Sci. Fis. Mat. Natur. Rend. Lincei
  (9) Mat. Appl.}, 14(3):207--217 (2004), 2003.
\newblock Renato Caccioppoli and modern analysis.

\bibitem{bre01}
Haim Brezis and Yanyan Li.
\newblock Topology and {S}obolev spaces.
\newblock {\em J. Funct. Anal.}, 183(2):321--369, 2001.

\bibitem{bri99}
Martin~R. Bridson and Andr{\'e} Haefliger.
\newblock {\em Metric spaces of non-positive curvature}, volume 319 of {\em
  Grundlehren der Mathematischen Wissenschaften [Fundamental Principles of
  Mathematical Sciences]}.
\newblock Springer-Verlag, Berlin, 1999.

\bibitem{bur84}
Francis~E. Burstall.
\newblock Harmonic maps of finite energy from noncompact manifolds.
\newblock {\em J. London Math. Soc. (2)}, 30(2):361--370, 1984.

\bibitem{cal56}
A.~P. Calder\'on and A.~Zygmund.
\newblock On singular integrals.
\newblock {\em Amer. J. Math.}, 78:289--309, 1956.

\bibitem{duo01}
Javier Duoandikoetxea.
\newblock {\em Fourier analysis}, volume~29 of {\em Graduate Studies in
  Mathematics}.
\newblock American Mathematical Society, Providence, RI, 2001.
\newblock Translated and revised from the 1995 Spanish original by David
  Cruz-Uribe.

\bibitem{eel88}
J.~Eells and L.~Lemaire.
\newblock Another report on harmonic maps.
\newblock {\em Bull. London Math. Soc.}, 20(5):385--524, 1988.

\bibitem{eel78}
James Eells and Luc Lemaire.
\newblock A report on harmonic maps.
\newblock {\em Bull. London Math. Soc.}, 10(1):1--68, 1978.

\bibitem{eel64}
James Eells, Jr. and Joseph~H. Sampson.
\newblock Harmonic mappings of {R}iemannian manifolds.
\newblock {\em Amer. J. Math.}, 86:109--160, 1964.

\bibitem{gol12}
Pawe\l Goldstein and Piotr Haj\l~asz.
\newblock Sobolev mappings, degree, homotopy classes and rational homology
  spheres.
\newblock {\em J. Geom. Anal.}, 22(2):320--338, 2012.

\bibitem{gra14}
Loukas Grafakos.
\newblock {\em Classical {F}ourier analysis}, volume 249 of {\em Graduate Texts
  in Mathematics}.
\newblock Springer, New York, third edition, 2014.

\bibitem{haj07}
Piotr Haj\l~asz.
\newblock Sobolev mappings: {L}ipschitz density is not a bi-{L}ipschitz
  invariant of the target.
\newblock {\em Geom. Funct. Anal.}, 17(2):435--467, 2007.

\bibitem{haj96}
Piotr Haj{\l}asz.
\newblock Sobolev spaces on an arbitrary metric space.
\newblock {\em Potential Anal.}, 5(4):403--415, 1996.

\bibitem{han03}
Fengbo Hang and Fanghua Lin.
\newblock Topology of {S}obolev mappings. {II}.
\newblock {\em Acta Math.}, 191(1):55--107, 2003.

\bibitem{hat02}
Allen Hatcher.
\newblock {\em Algebraic topology}.
\newblock Cambridge University Press, Cambridge, 2002.

\bibitem{hei01}
Juha Heinonen.
\newblock {\em Lectures on analysis on metric spaces}.
\newblock Universitext. Springer-Verlag, New York, 2001.

\bibitem{hei98}
Juha Heinonen and Pekka Koskela.
\newblock Quasiconformal maps in metric spaces with controlled geometry.
\newblock {\em Acta Math.}, 181(1):1--61, 1998.

\bibitem{HKST07}
Juha Heinonen, Pekka Koskela, Nageswari Shanmugalingam, and Jeremy Tyson.
\newblock {\em Sobolev spaces on metric measure spaces: an approach based on
  upper gradients}.
\newblock New Mathematical Monographs. Cambridge University Press, United
  Kingdom, first edition, 2015.

\bibitem{jost94}
J{\"u}rgen Jost.
\newblock Equilibrium maps between metric spaces.
\newblock {\em Calc. Var. Partial Differential Equations}, 2(2):173--204, 1994.

\bibitem{jost95}
J{\"u}rgen Jost.
\newblock Convex functionals and generalized harmonic maps into spaces of
  nonpositive curvature.
\newblock {\em Comment. Math. Helv.}, 70(4):659--673, 1995.

\bibitem{jost96}
J{\"u}rgen Jost.
\newblock Generalized harmonic maps between metric spaces.
\newblock In {\em Geometric analysis and the calculus of variations}, pages
  143--174. Int. Press, Cambridge, MA, 1996.

\bibitem{jost97}
J{\"u}rgen Jost.
\newblock Generalized {D}irichlet forms and harmonic maps.
\newblock {\em Calc. Var. Partial Differential Equations}, 5(1):1--19, 1997.

\bibitem{kei03}
Stephen Keith.
\newblock Modulus and the {P}oincar\'e inequality on metric measure spaces.
\newblock {\em Math. Z.}, 245(2):255--292, 2003.

\bibitem{koh57}
J.~J. Kohn and D.~C. Spencer.
\newblock Complex {N}eumann problems.
\newblock {\em Ann. of Math. (2)}, 66:89--140, 1957.

\bibitem{kor93}
Nicholas~J. Korevaar and Richard~M. Schoen.
\newblock Sobolev spaces and harmonic maps for metric space targets.
\newblock {\em Comm. Anal. Geom.}, 1(3-4):561--659, 1993.

\bibitem{mal03}
Jan Mal\'y, David Swanson, and William~P. Ziemer.
\newblock The co-area formula for {S}obolev mappings.
\newblock {\em Trans. Amer. Math. Soc.}, 355(2):477--492, 2003.

\bibitem{nak93}
Nobumitsu Nakauchi.
\newblock Homomorphism between homotopy groups induced by elements of the
  {S}obolev space {$L^{1,p}(M,N)$}.
\newblock {\em Manuscripta Math.}, 78(1):1--7, 1993.

\bibitem{pig09}
Stefano Pigola and Giona Veronelli.
\newblock On the homotopy class of maps with finite {$p$}-energy into
  non-positively curved manifolds.
\newblock {\em Geom. Dedicata}, 143:109--116, 2009.

\bibitem{uhl83}
Richard Schoen and Karen Uhlenbeck.
\newblock Boundary regularity and the {D}irichlet problem for harmonic maps.
\newblock {\em J. Differential Geom.}, 18(2):253--268, 1983.

\bibitem{sch79}
Richard Schoen and Shing~Tung Yau.
\newblock Compact group actions and the topology of manifolds with nonpositive
  curvature.
\newblock {\em Topology}, 18(4):361--380, 1979.

\bibitem{see59}
R.~T. Seeley.
\newblock Singular integrals on compact manifolds.
\newblock {\em Amer. J. Math.}, 81:658--690, 1959.

\bibitem{sha00}
Nageswari Shanmugalingam.
\newblock Newtonian spaces: an extension of {S}obolev spaces to metric measure
  spaces.
\newblock {\em Rev. Mat. Iberoamericana}, 16(2):243--279, 2000.

\bibitem{teri2}
Elefterios Soultanis.
\newblock {Existence of {p}-energy minimizers in homotopy classes and lifts of
  Newtonian maps}.
\newblock {\em To appear in J. Anal. Math.}
\newblock arXiv:1506.07767 [math.DG].

\bibitem{teri1}
Elefterios Soultanis.
\newblock {Homotopy classes of Newtonian spaces}.
\newblock {\em To appear in Rev. Mat. Iberoamericana}.
\newblock arXiv:1309.6472 [math.MG].

\bibitem{ver12}
Giona Veronelli.
\newblock A general comparison theorem for {$p$}-harmonic maps in homotopy
  class.
\newblock {\em J. Math. Anal. Appl.}, 391(2):335--349, 2012.

\bibitem{whi86}
Brian White.
\newblock Infima of energy functionals in homotopy classes of mappings.
\newblock {\em J. Differential Geom.}, 23(2):127--142, 1986.

\bibitem{whi88}
Brian White.
\newblock Homotopy classes in {S}obolev spaces and the existence of energy
  minimizing maps.
\newblock {\em Acta Math.}, 160(1-2):1--17, 1988.

\end{thebibliography}
\end{document}